\documentclass[leqno,12pt]{article} 
 \usepackage{amsmath, amssymb}
 \usepackage{amsthm} 
\newtheorem{thm}{Thm}[section]
\newtheorem{cor}[thm]{Corollary}
\newtheorem{lem}[thm]{Lemma}
\newtheorem{prop}[thm]{Proposition}

\theoremstyle{definition}
\newtheorem{defn}[thm]{Definition}

\theoremstyle{remark}

\numberwithin{equation}{section}
\newcommand{\addsupport}[1]{} 

\newcommand{\bsm}{\left(\begin{smallmatrix}}
\newcommand{\esm}{\end{smallmatrix}\right)}

\begin{document}

\title{ MORE PROPERTIES OF THE INCOMPLETE GAMMA FUNCTIONS}
\author{R. AlAhmad\\ Mathematics Department, Yarmouk University, \\Irbid, Jordan 21163,\\ email:rami\_thenat@yu.edu.jo}
\date{\today}%
\maketitle
\begin{abstract}
In this paper, additional properties of the lower gamma functions and the error functions are introduced and proven. In particular, we prove interesting relations between the error functions and Laplace transform. \\
{\bf AMS Subject Classification}: 33B20\\
{\bf Key Words and Phrases}: Incomplete beta and gamma functions, error functions.
\end{abstract}
\section{Introduction}
\begin{defn}The lower incomplete gamma function is defined as:
\[ \gamma(s,x) = \int_0^x t^{s-1}e^{-t}dt .\]
\end{defn}
Clearly,  $\gamma(s,x)\longrightarrow\Gamma(s)$ as $x\longrightarrow\infty$.
The properties of these functions are listed in many references( for example see~\cite{tables},~\cite{Miller} and ~\cite{Tem}). In particular, the following properties are needed:
\begin{prop}\label{prop1}~\cite{ape}

\begin{enumerate}
\item $\gamma(s+1,x) =s\gamma(s,x) - x^{s} e^{-x},$
\item$\gamma(1,x) = 1 - e^{-x},$
\item$\gamma\left(\frac{1}{2}, x\right) = \sqrt\pi\operatorname{erf}\left(\sqrt x\right).$
\end{enumerate}
\end{prop}
\section{Main Results}
\begin{prop}\cite{ape}\label{ape}For $a<0$ and $ a+b>0$\[\int_{0}^{\infty}x^{a-1}\mathop{\gamma\/}\nolimits\!\left(b,x\right)dx=-\frac{%
\mathop{\Gamma\/}\nolimits\!\left(a+b\right)}{a}.\]
\end{prop}
\begin{prop}\label{prop3} For $a\neq0$
$$\int_0^{\sqrt{t}} r^se^{-(ar)^2}dr=\frac{1}{2a^{s+1}}\gamma(\frac{s+1}{2},a^2t).$$
\end{prop}
\begin{proof}The substitution $u=r^2$ gives\begin{align*}\int_0^{\sqrt{t}} r^se^{-(ar)^2}dr=\frac{1}{2a^{s+1}}\int_{0}^{a^2t} e^{-u} u^{\frac{s-1}{2}}du=\frac{1}{2a^{s+1}}\gamma(\frac{s+1}{2},a^2t).
\end{align*}
\end{proof}
\begin{prop}\label{main2}
$$ \eta^{x}\int_{0}^{\xi}t^{x-1}e^{-\eta t}dt =\gamma(x,\eta\xi)$$
\end{prop}

\begin{proof}\begin{align*}\gamma(x,\eta\xi)&=\int_{0}^{\eta\xi} e^{-t} t^{x-1}\\
&=\int_{0}^{\xi}e^{-\eta t}(\eta t)^{x-1}=\eta^{x}\int_{0}^{\xi}t^{x-1}e^{-\eta t}dt
\end{align*}
\end{proof}
The incomplete gamma function $\gamma(a,t)$ satisfies
\begin{lem}\label{main7} For $t\geq0$,
$$ \gamma(a,t)=\frac{2}{\Gamma(b)}\int_0^{\pi/2}\gamma(a+b,t\sec^2(\theta))\cos^{2a-1}(\theta)\sin^{2b-1}(\theta)d\theta.$$
\end{lem}
\begin{proof} The strip $\{(x,y):0< x< \sqrt{t}, y> 0\}$ is mapped to $\{(r,\theta):0<r<\sqrt{t}\sec(\theta),0<\theta<\pi/2\}$, where $x=r\cos(\theta)$ and $y=r\sin(\theta)$.

Now, Using the substitutions $z=x^2$ and $v=y^2$, we get
$$ 4\int_{0}^{\sqrt{t}}\int_{0}^{\infty} e^{-{x^2}} x^{2a-1}e^{-{y^2}}y^{2b-1}dydx=\int_{0}^{t}\int_{0}^{\infty} e^{-z} z^{a-1}e^{-v}v^{b-1}dvdz. $$ This result gives
\begin{align}\label{Eq2}\gamma(a,t)\Gamma(b)&=\int_{0}^{t}\int_{0}^{\infty} e^{-z} z^{a-1}e^{-v}v^{b-1}dvdz\nonumber \\
&= 4\int_{0}^{\sqrt{t}}\int_{0}^{\infty} e^{-{x^2}} x^{2a-1}e^{-{y^2}}y^{2b-1}dydx\nonumber \\
&=4\int_0^{\pi/2}\int_0^{\sqrt{t}\sec(\theta)}\cos^{2a-1}(\theta)\sin^{2b-1}(\theta)e^{-r^2}r^{2a+2b-1}drd\theta.\end{align}
Using Proposition~\ref{prop3} and Equation~\ref{Eq2} we get
  $$ \gamma(a,t)\Gamma(b)=2\int_0^{\pi/2}\gamma(a+b,t\sec^2(\theta))\cos^{2a-1}(\theta)\sin^{2b-1}(\theta)d\theta.$$
  \end{proof}

As a remark, in Lemma~\ref{main7}, if we let $t\longrightarrow\infty$ then $\gamma(a,t)\longrightarrow\Gamma(a)$ and $\gamma(a+b,t\sec^2(\theta))\longrightarrow \Gamma(a+b)$. This proves the well-known result that is\begin{equation}\label{known}\beta(a,b)=\frac{\Gamma(a)\Gamma(b)}{\Gamma(a+b)}=2\int_0^{\pi/2}\cos^{2a-1}(\theta)\sin^{2b-1}(\theta)d\theta.\end{equation}
As a manner of fact, one can easily show that

 It is Known that the gamma function satisfies the property
\begin{prop} The gamma function $\Gamma$ satisfies
$$\Gamma(1-a) \Gamma(a) = \pi \csc(\pi a).$$
\end{prop}
The following is an extension of this result
\begin{lem}\label{mainr1}
For $t\geq0$,
\begin{align*}
\gamma(a,t)\Gamma(1-a)&=2\int_0^{\pi/2}(1-e^{-t\sec^2(\theta)})\cot^{2a-1}(\theta)d\theta\\=&\pi\csc(\pi a)-2\int_0^{\pi/2}e^{-t\sec^2(\theta)}\cot^{2a-1}(\theta)d\theta.\end{align*}
\end{lem}
\begin{proof}
Proposition~\ref{prop1} implies \begin{equation}\label{main8}\gamma(1,t\sec^2(\theta))=1-e^{-t\sec^2(\theta)}.\end{equation} Then Lemma~\ref{main7} with $b=1-a$ gives the result.
\end{proof}
This lemma gives the following
\begin{prop} For $-1<a<0$,\[\int_0^\infty\frac{1-e^{-t}}{t^{1-a}}dt=-\frac{\Gamma(a+1)}{a}.\]

\end{prop}
\begin{proof}Using Lemma~\ref{ape} we get\[\int_{0}^{\infty}x^{a-1}\mathop{\gamma\/}\nolimits\!\left(b,x\right)dx=-\frac{%
\mathop{\Gamma\/}\nolimits\!\left(a+b\right)}{a}.\]
Multiply both sides by $\Gamma(1-b)$ and use Lemma~\ref{mainr1} to get
\[-\frac{%
\Gamma(1-b)\mathop{\Gamma\/}\nolimits\!\left(a+b\right)}{a}=2\int_0^\infty\int_0^{\pi/2}(1-e^{-t\sec^2(\theta)})t^{a-1}\cot^{2a-1}(\theta)d\theta dt.\]
Interchange the order of the integral and use the substitution $u=t\sec^2(\theta)$ to get:
\[-\frac{%
\Gamma(1-b)\mathop{\Gamma\/}\nolimits\!\left(a+b\right)}{a}=\left(\int_0^\infty\frac{1-e^{-u}}{u^{1-a}}du\right)\left(2\int_0^{\pi/2}\cos^{2a+2b-1}(\theta)\sin^{1-2b}(\theta)d\theta d\theta\right).\]
Use Equation~\ref{known} to get the result.
\end{proof}

\begin{cor}
The error function satisfies $$\operatorname{erf}(\sqrt{at})=\frac{2}{\sqrt{\pi}\Gamma(b)}\int_{0}^{\pi/2}\gamma(\frac{1}{2}+b,at\sec^2(\theta))\sin^{2b-1}(\theta)d\theta.$$In particular,$$\operatorname{erf}(\sqrt{at})=1-\frac{2}{\pi}\int_{0}^{\pi/2}e^{-at\sec^2(\theta)}d\theta.$$
\end{cor}
\begin{proof}By Lemma~\ref{main7} with $a=1/2$ and Proposition~\ref{prop1}, we get
$$ \sqrt{\pi}\operatorname{erf}(\sqrt{at})=\gamma(1/2,at)=\frac{2}{\Gamma(b)}\int_0^{\pi/2}\gamma(\frac{1}{2}+b,at\sec^2(\theta))\sin^{2b-1}(\theta)d\theta.$$
Therefore, $$\operatorname{erf}(\sqrt{at})=\frac{2}{\sqrt{\pi}\Gamma(b)}\int_0^{\pi/2}\gamma(\frac{1}{2}+b,at\sec^2(\theta))\sin^{2b-1}(\theta)d\theta.$$

The last equation with $b=1$ gives
$$\operatorname{erf}(\sqrt{at})=\frac{2}{\pi}\int_0^{\pi/2}\gamma(1,at\sec^2(\theta))d\theta.$$
Using Equation~\ref{main8} implies
  $$\operatorname{erf}(\sqrt{at})=\frac{2}{\pi}\int_0^{\pi/2}\gamma(1,at\sec^2(\theta))d\theta=1-\frac{2}{\pi}\int_{0}^{\pi/2}e^{-at\sec^2(\theta)}d\theta.$$
i.e., \begin{equation}\label{main9}\operatorname{erfc}(\sqrt{at})=1-\operatorname{erf}(\sqrt{at})=\frac{2}{\pi}\int_{0}^{\pi/2}e^{-at\sec^2(\theta)}d\theta.\end{equation}
\end{proof}
Multiplying Equation~\ref{main9} by $e^{at}$ implies
\begin{equation}\label{main91}e^{at}\operatorname{erfc}(\sqrt{at})=\frac{2}{\pi}\int_{0}^{\pi/2}e^{-at\tan^2(\theta)}d\theta.\end{equation}
Equation~\ref{main9} gives the following results
\begin{lem}\label{main10}
\begin{equation}\label{main101}\int_0^\infty f(t) \operatorname{erfc}(a\sqrt{t})dt
=\frac{2}{\pi}\int_0^{\pi/2}F(a^2\sec^2(\theta))d\theta=
\frac{2a}{\pi}\int_a^\infty\frac{F(s^2)}{s\sqrt{s^2-a^2}}ds,\end{equation} where $F(s)=(\mathcal{L}f)(s)$ is the Laplace transform of $f(t)$.
\end{lem}
\begin{proof}
Equation~\ref{main9} implies
$$ \int_0^\infty f(t)\operatorname{erfc}(a\sqrt{t})dt=\frac{2}{\pi}\int_0^\infty\int_{0}^{\pi/2} f(t) e^{-ta^2\sec^2(\theta)}d\theta dt.$$ By interchanging the order of the integral, we get:
  \begin{align}\label{main13} \int_0^\infty f(t)\operatorname{erfc}(a\sqrt{t})dt=\frac{2}{\pi}\int_{0}^{\pi/2}\int_0^\infty f(t) e^{-ta^2\sec^2(\theta)}dt d\theta .\end{align} Now, Since the Laplace transform is given as $F(s)=(\mathcal{L}f)(s)= \int_0^\infty f(t) e^{-st}dt$, then
  \begin{equation}\label{main14}\int_0^\infty f(t) e^{-ta^2\sec^2(\theta)}dt=F(a^2\sec^2(\theta)).\end{equation}
    Then Equations~\ref{main13} and~\ref{main14} prove the result that $$\int_0^\infty f(t) \operatorname{erfc}(a\sqrt{t})dt=\frac{2}{\pi}\int_0^{\pi/2}F(a^2\sec^2(\theta))d\theta.$$ In addition, the substitution $s=a\sec^2\theta$ implies that $$\frac{2}{\pi}\int_0^{\pi/2}F(a^2\sec^2(\theta))d\theta=\frac{2a}{\pi}\int_a^\infty\frac{F(s^2)}{s\sqrt{s^2-a^2}}ds.$$

\end{proof}
\begin{cor} For $r>-1$,
 \[\int_0^\infty t^r \operatorname{erfc}(a\sqrt{t})dt=\frac{\Gamma(r+\frac{3}{2})}{a^{2r+2}\sqrt{\pi}(1+r)}=\frac{2a\Gamma(r+1)}{\pi}\int_a^\infty\frac{1}{s^{2r+3}\sqrt{s^2-a^2}}ds.\] Moreover, for $ \mu>-1$
\[\int_0^\infty t^\mu \operatorname{erfc}(at)dt=\frac{1}{\sqrt{\pi}}\frac{\Gamma(1+\frac{\mu}{2})}{a^{\mu+1}(1+\mu)}.\]
\end{cor}
\begin{proof}
Lemma~\ref{main10} implies that \[\int_0^\infty t^r \operatorname{erfc}(a\sqrt{t})dt=\frac{2}{\pi}\int_{0}^{\pi/2} \frac{\Gamma(r+1)}{(a^2\sec^2\theta)^{r+1}}d\theta.\] Use Equation~\ref{known} to get that
$$\int_{0}^{\pi/2} \frac{1}{(a^2\sec^2\theta)^{r+1}}d\theta=\frac{1}{a^{2r+2}}\int_{0}^{\pi/2} (\cos^2\theta)^{r+1}d\theta=\frac{1}{a^{2r+2}}\frac{\Gamma(r+\frac{3}{2})\sqrt{\pi}}{2\Gamma(r+2)}.$$ Therefore,
\begin{equation}\label{main14}\int_0^\infty t^r \operatorname{erfc}(a\sqrt{t})dt=\frac{2}{\pi}\Gamma(r+1)\frac{\Gamma(r+\frac{3}{2})\sqrt{\pi}}{2a^{2r+2}\Gamma(r+2)}
=\frac{\Gamma(r+\frac{3}{2})}{a^{2r+2}\sqrt{\pi}(1+r)}.\end{equation}  Now, the substitution $u=\sqrt{t}$ into equation~\ref{main14} with $a=1$ gives that
$$2\int_0^\infty u^{2r+1}\operatorname{erfc}(u)du=\int_0^\infty t^r \operatorname{erfc}(\sqrt{t})dt=\frac{\Gamma(r+\frac{3}{2})}{\sqrt{\pi}(1+r)}.$$
Hence,
$$\int_0^\infty u^{\mu}\operatorname{erfc}(u)du=\frac{1}{2}\frac{\Gamma(\frac{\mu-1}{2}+\frac{3}{2})}{\sqrt{\pi}(1+\frac{\mu-1}{2})}=\frac{1}{\sqrt{\pi}}\frac{\Gamma(1+\frac{\mu}{2})}{1+\mu}.$$
Lastly,
 \[\int_0^\infty t^\mu \operatorname{erfc}(at)dt=\frac{1}{a^{\mu+1}}\int_0^\infty t^\mu \operatorname{erfc}(t)dt=\frac{1}{\sqrt{\pi}}\frac{\Gamma(1+\frac{\mu}{2})}{a^{\mu+1}(1+\mu)}.\]
\end{proof}
Following similar proof of Lemma~\ref{main10}  and using Equation~\ref{main91}, we get the following result
\begin{lem}\label{main15}
\begin{equation}\label{main16}\int_0^\infty e^{a^2t}f(t) \operatorname{erfc}(a\sqrt{t})dt=\frac{2}{\pi}\int_0^{\pi/2}F(a^2\tan^2(\theta))d\theta
=\frac{2a}{\pi}\int_a^\infty\frac{F(s^2)}{s^2+a^2}ds,\end{equation} where $F(s)=(\mathcal{L}f)(s)$ is the Laplace transform of $f(t)$.
\end{lem}
\begin{cor} For $r>-1$, $$\int_0^\infty t^re^{a^2t}\operatorname{erfc}(a\sqrt{t})dt=\frac{2a\Gamma(r+1)}{\pi}\int_a^\infty\frac{1}{s^{2r+2}(s^2+a^2)}ds.$$

\end{cor}
Substitute the Laplace transform of the dirac delta function $f(t)=\delta(t-b)$ which is $F(s)=e^{-bs}$ into Equations~\ref{main101} and~\ref{main16} to get that
\begin{prop}
  \begin{equation}\label{main17}\int_a^\infty\frac {e^{-b^2t^2}}{t\sqrt{t^2-a^2}}dt=\frac{\pi}{2a}\operatorname{erfc}(ab).\end{equation}
  Also, \begin{equation}\label{main18}\int_0^\infty\frac {e^{-b^2t^2}}{t^2+a^2}dt=\frac{\pi}{2a}e^{a^2b^2}\operatorname{erfc}(ab).\end{equation}
  \end{prop}

\end{document}